\documentclass{jgcc}

\pdfoutput=1

\usepackage{lastpage}

\jgccdoi{16}{1}{1}{13493} 

\jgccheading{}{\pageref{LastPage}}{}{}{Mar.~14,~2024}{May.~13,~2024}{}   

\usepackage[english,russian]{babel}
\usepackage[T2A]{fontenc}
\usepackage[cp1251]{inputenc}
\usepackage{amsfonts}
\usepackage{amssymb, amsthm, amscd}
\usepackage{amsmath}
\usepackage{mathtools}
\usepackage{needspace}
\usepackage{etoolbox}
\usepackage{lipsum}
\usepackage{comment}
\usepackage{cmap}
\usepackage{euscript}
\usepackage{epigraph}
\usepackage[numbers]{natbib}
\usepackage{hyperref}

\newcommand{\Z}{\mathbb{Z}}
\newcommand{\N}{\mathbb{N}}

\newcommand{\Q}{\mathbb{Q}}

\renewcommand{\O}{\EuScript{O}}

\newcommand{\M}{\mathfrak{M}}

\newcommand{\ovl}{\overline}

\newcommand{\ph}{\varphi}

\newcommand{\sub}{\subseteq}

\renewcommand{\ge}{\geqslant}
\renewcommand{\le}{\leqslant}
\newcommand{\sm}{\setminus}

\DeclareMathOperator{\ac}{ac}
\DeclareMathOperator{\tp}{tp}
\DeclareMathOperator{\trdeg}{tr.deg}
\DeclareMathOperator{\dist}{dist}

\DeclareMathOperator{\SL}{SL}

\def\le{\leqslant}

\def\ge{\geqslant}

\title{On countable isotypic structures.}


\author{Pavel Gvozdevsky}
\date{}
\address{Department of Mathematics, Bar-Ilan University, 5290002 Ramat Gan, ISRAEL}
\thanks{The paper is written as part of the author's post-doctoral fellowship at Bar-Ilan University, Department of Mathematics; and is supported by ISF grant 1994/20. }
\email{gvozdevskiy96@gmail.com}

\begin{document}
\selectlanguage{english}

\begin{abstract}
	We obtain several results concerning the concept of isotypic structures. Namely we prove that any field of finite transcendence degree over a prime subfield is defined by types; then we construct isotypic but not isomorphic structures with countable underlying sets: totally ordered sets, fields, and  groups. This answers an old question by B. Plotkin for groups.
\end{abstract}

\maketitle

\epigraph{To the memory of Ben Fine}

\section{Introduction}

The concept of isotypicity for structures  naturally arose within the framework of universal algebraic geometry and logical geometry \cite{Pl_7_lecturs}, \cite{MRem}, \cite{DMR}. In the paper \cite{PZh} it was formulated as a kind of Morita-type logical equivalence on algebras. Namely,

\begin{defi}  Algebras $H_1$ and $H_2$ are logically similar if the categories of
definable sets over $H_1$ and $H_2$  are isomorphic.\end{defi}

\smallskip

This definition is related to the following one formulated by B.Plotkin in \cite{Pl_Iso} for algebras:

\begin{defi} (see ~\cite{Pl_Iso}). Let $\mathcal L$ be a first order language and $A$, $B$ be $\mathcal L$- structures. Then $A$ and $B$ are isotypic if for any finite tuple $\bar a=(a_1,\ldots, a_k)$ over $A$ there exists a tuple $\bar b=(b_1,\ldots, b_k)$ over $B$ such that their types coincide, that is,  $\tp_A(\bar a)=\tp_B(\bar b)$, and vice versa, where for a $k$-tuple $\bar a$ we define $\tp_A(\bar a)$ as a set of all first order formulas in $k$ variables that hold true for $\bar a$.\end{defi}

Isotypic algebras are discussed in \cite{MyasRomTypes},\cite{Pl_Gagta}, \cite{PAP}, \cite{PP}, \cite{PZh}, \cite{Zh}. In particular, in this paper we follow the definition from \cite{MyasRomTypes}:

\begin{defi}(see~\cite{MyasRomTypes}) We say that a structure $A$ is {\it defined by types} if any structure $B$ isotypic to $A$ is isomorphic to $A$.\end{defi}

\smallskip

The major problem behind all the considerations which remain widely open is whether two finitely generated isotypic groups are isomorphic, see \cite{Pl_Iso}. Our paper can be viewed as a step towards the clarification of this problem. It also answers the question about existence of countable isotypic but not isomorphic groups. Note that it was recently found that there are such groups in any cardinality bigger than contable, see~\cite[Remark 4.1]{Bunina-isotypic}.

The paper is structured as follows. In Section~\ref{FieldsSec} we prove that any field of finite transcendence degree over a prime subfield is defined by types. In Section~\ref{OrdersSec} we give an example of countable isotypic but not isomorphic totally ordered sets. In Section~\ref{FieldsGroupsSec} we construct countable isotypic but not isomorphic fields and groups.

\section*{Acknowledgment}

I am grateful to the anonymous reviewer of the first version of this paper, who pointed out references that allow to simplify the proofs, and suggested a way to extend one of the results from rings to fields. 

I am also thankful to ISF grant 1994/20 for the support of the paper.

\section{Isotypic fields of finite transcendence degree are isomorphic}
\label{FieldsSec}

In this Section we prove the following theorem.

\begin{thm}
 Any field of finite transcendence degree over a prime subfield is defined by types.
\end{thm}
\begin{proof}
	We will denote the transcendence degree of a field $K$ by $\trdeg K$.
	
	Let $E$ be a field of finite transcendence degree over a prime subfield, and $F$ be a field that is isotypic to $E$. Let us prove that $E\simeq F$.
	
	First note that $E$ and $F$ are elementary equivalent; hence they have the same characteristic, and hence the same prime subfield $k$. Let $\trdeg E=n$. Let $\ovl{a}=(a_1,\ldots, a_n)\in E^n$ be a transcendence basis of $E$. Then there exists $\ovl{b}=(b_1,\ldots,b_n)\in F^n$ such that $\tp_{E}(\ovl{a})=\tp_{F}(\ovl{b})$. Then the elements $b_i$ do not satisfy any nontrivial equation with integer coefficients, hence they are independent over $k$. Thus $\trdeg(F)\ge \trdeg (E)$. Similarly, if $F$ has $n+1$ algebraically independent elements, then so does $E$. Therefore, $\trdeg(F)\le \trdeg (E)$. Therefore, $\trdeg E=\trdeg F=n$; hence the elements $b_i$ form a transcendence basis of $F$.
	
	We may assume that $a_i=b_i$ for all $i$; so that $E$ and $F$ share the common subfield $K=k(a_1,\ldots,a_n)$, and both fields are algebraic over $K$. Since $\tp_{E}(\ovl{a})=\tp_{F}(\ovl{a})$, it follows that any polynomial $f\in K[x]$ has a root in $E$ if and only if it has a root in $F$ (existence of a root is an existential formula on $\ovl{a}$ with one quantifier). Hence by~\cite{FriedJarden}[Lemma 20.6.3 (b)], we have $E\simeq F$.
\end{proof}

\section{An example of countable isotypic but not isomorphic totally ordered sets.}
\label{OrdersSec}

In this Section we prove the following theorem.

\begin{thm} There exist two countable isotypic but not isomorphic totally ordered sets.\end{thm}

\smallskip

See Corollary~\ref{ExmpOrders} below for the concrete example.

\smallskip

First we need some preparation.

\begin{defi}
	Let $(A,<)$ be a totally ordered set. Let $x,y\in A$. Set
	\[
	\dist(x,y)=\begin{cases}
		0,\quad &\text{ if }x=y,\\
		|\{z\colon x<z<y\}|+1,\quad &\text{ if }x\ne y.
	\end{cases}
	\]
	So $\dist(x,y)$ is an element of the set $\N_0\cup\{\infty\}$. \end{defi}

For totally ordered sets $(A,<_A)$ and $(B,<_B)$ we denote by $(A\times B,<_{AB})$ the Cartesian product $A\times B$ with the lexicographic order
\[
(a,b)<_{AB}(a',b')\Leftrightarrow (a<a' \vee (a=a'\wedge b<b')).
\]

Let $(\Z,<_{\Z})$ be the set of integers with the usual order. Note that for a  totally ordered set $A$, if we consider the product $(A\times\Z,<_{A\Z})$, we have
\[
\dist((a,k),(a'k,))=\begin{cases}
	\infty,\quad &\text{ if }a\ne a',\\
	|k-k'|,\quad &\text{ if }a= a'.
\end{cases}
\]

\begin{prop} Let $(A,<_A)$, $(B,<_B)$ be totally ordered sets. Set $(A',<_{A'})=(A\times\Z,<_{A\Z})$ and $(B',<_{B'})=(B\times\Z,<_{B\Z})$. Let $\ovl{a}=(a_1,\ldots,a_n)\in (A')^n$ and $\ovl{b}=(b_1,\ldots,b_n)\in (B')^n$ be such that, firstly, $a_1\le_{A'}\ldots\le_{A'}a_n$, and $b_1\le_{B'}\ldots\le_{B'}b_n$; and secondly, for any $1\le i\le n-1$ we have $\dist_{A'}(a_i,a_{i+1})=\dist_{B'}(b_i,b_{i+1})$. Then $\tp_{A'}(\ovl{a})=\tp_{B'}(\ovl{b})$.\end{prop}

\begin{proof}
	To prove that $\tp_{A'}(\ovl{a})=\tp_{B'}(\ovl{b})$ is the same as prove that $(A',<_{A'},\ovl{a})$ and $(B',<_{B'},\ovl{b})$ are elementary equivalent in the language of ordered sets with $n$ constants (constant can be understand as an unary predicate indicating that given element is this constant). So we can do it by showing that for any $N\in\N$ in the corresponding Ehrenfeucht-Fraisse game of length $N$ the duplicator has a winning strategy.
	
	Fix the number $N$. By \cite[Theorem 1.8]{Poizat} the tuples $\ovl{a}$ and $\ovl{b}$ are $N$-equivalent (see definition in \cite[Section 1.1.]{Poizat}) in the language of ordered sets. Let $a_{n+k}\in A'$ and $b_{n+k}\in B'$ be the elements chosen at $k$-th turn of the game. The duplicator's winning strategy is to make choice in such a way that the tuples $(a_1,\ldots, a_{n+k})$ and $(b_1,\ldots, b_{n+k})$ be $(N-k)$-equivalent. He can do it by the definition of $N$-equivalence. In this case, in the end of the game the tuples $(a_1,\ldots, a_{n+N})$ and $(b_1,\ldots, b_{n+N})$ will be $0$-equivalent, i.e the order on them will be the same. In particular that means that $(a_{n+1},\ldots, a_{n+N})$ and $(b_{n+1},\ldots, b_{n+N})$ are in the same order and for any $1\le k\le n$ and $n+1\le l\le n+N$ we have $a_k=a_l\Leftrightarrow b_k=b_l$, which means that the duplicator won. 
\end{proof}

The proposition above allows us to construct many examples of isotypic totally ordered sets. Namely we have the following corollary.

\begin{cor} Let $(A,<_A)$ and $(B,<_B)$ be infinite totally ordered sets. Then the totally ordered sets $(A\times\Z,<_{A\Z})$ and $(B\times\Z,<_{B\Z})$ are isotypic.\end{cor}
\begin{proof}
	It follows from the fact that if a totally ordered set $(A,<_A)$ is infinite, then for every $(d_1,\ldots,d_{n-1})\in \left(\N_0\cup\{\infty\}\right)^{n-1}$ there exists $(a_1,\ldots,a_n)\in (A\times\Z)^n$ such that $\dist_{A\times\Z}(a_i,a_{i+1})=d_i$ for all $1\le i\le n-1$.
\end{proof}

\begin{rem}It is well known (see for example~\cite[Proposition~2.4.10]{MarkerBook}) that for any $(A,<_A)$ and $(B,<_B)$ the totally ordered sets $(A\times\Z,<_{A\Z})$ and $(B\times\Z,<_{B\Z})$ are elementary equivalent.\end{rem}

\smallskip

The following lemma shows that these isotypic totally ordered sets are not isomorphic to each other.

\begin{lem} Let $(A,<_A)$ and $(B,<_B)$ be totally ordered sets such that $(A\times\Z,<_{A\Z})\simeq (B\times\Z,<_{B\Z})$. Then $(A,<_A)\simeq (B,<_B)$.\end{lem}
\begin{proof}
	It follows from the fact that $(A,<_A)$ is the quotient of $(A\times\Z,<_{A\Z})$ by equivalence relation $a\sim a'\Leftrightarrow \dist(a,a')<\infty$; and this quotient inherits the order.
\end{proof}

Now we can construct our example.

\begin{cor} \label{ExmpOrders} The pair $(\Z\times\Z,<_{\Z\Z})$ and $(\Q\times\Z,<_{\Q\Z})$ is an example of countable isotypic but not isomorphic totally ordered set.\end{cor}

\section{Examples of countable isotypic but not isomorphic fields and groups.}
\label{FieldsGroupsSec}

In this Section we prove the following theorems.

\begin{thm} There exist two countable isotypic but not isomorphic rings.\end{thm}

\smallskip

See Corollary~\ref{ExmpFields} below for the concrete example.

\begin{thm} There exist two countable isotypic but not isomorphic groups.\end{thm}

\smallskip

See Corollary~\ref{ExmpGroups} below for the concrete example.

\smallskip

Here is how our construction goes.

\begin{defi}For a totally ordered set $(A,<_A)$ we define the ordered abelian group $\Gamma_A$ as $ \bigoplus_{a\in A}(\Q,+)$ with the lexicographic order (the positive elements are those with the highest non-zero component being positive). Then we construct the valued field $K_A$ as follows: take the field $\Q(x_a\colon a\in A)$, where $x_a$ are algebraically independent variables; add all the rational powers of the variables $x_a$; define the non-archimedean valuation $v$ with values in $\Gamma_A$ to be zero on $\Q\sm\{0\}$ and such that for all $a\in A$ and $q\in \Q$ the valuation $v(x_a^q)$ has $q$ in the component that correspond to $a$ and zero in all the other components; define $K_A$ to be the henselization of this field with respect to the valuation $v$, with the natural extension of that valuation (for the notion of henselization see, for example,~\cite[Section 5.2]{EnglerPrestel}).   
 \end{defi}

\begin{prop} Let $(A,<_A)$, $(B,<_B)$ be infinite totally ordered sets. Then the fields $K_A$ and $K_B$ are isotypic.\end{prop}

\begin{proof}
	Take $\ovl{\xi}=(\xi_1,\ldots,\xi_n)\in K_A^n$ and let us prove that there exists $\ovl{\zeta}=(\zeta_1,\ldots,\zeta_n)\in K_B^n$ such that $\tp_{K_A}(\ovl{\xi})=\tp_{K_B}(\ovl{\zeta})$.
	
	Each element $\xi_i$ is a root of some polynomial over $\Q(x_a\colon a\in A)$. Each of these polynomials has finitely many coefficients, and for each coefficient there are finitely many elements $a\in A$ such that the variable $x_a$ is involved. Let $S=(a_1,\ldots,a_N)\sub A$ be the set of all these elements, and let $a_1<\ldots<a_N$.
	
	The field $K_A$ has a unique valuation preserving embedding in the completion of the field $\Q(x_a^q\colon a\in A,\; q\in\Q)$. Thus each of the elements $\xi_i$ can be presented as a power series in $x_a^q$ and clearly those series involve only $x_a^q$ with $a\in S$. Now choose elements $b_1,\ldots,b_N\in B$ such that $b_1<\ldots<b_N$. Then define $\zeta_i$ to be the same power series as $\xi_i$ with each $x_{a_k}$ replaced by $x_{b_k}$. Note that $\zeta_i$ are algebraic over $\Q(x_b\colon b\in B)$ and hence belong to $K_B$, because the residue field has characteristic zero and hence by~\cite[Theorem 4.1.10]{EnglerPrestel} the henselization is algebraically maximal. 
	
	Now let us prove that $\tp_{K_A}(\ovl{\xi})=\tp_{K_B}(\ovl{\zeta})$. Consider the triples $(K_A,\Gamma_A,\Q)$ and $(K_B,\Gamma_B,\Q)$ as structures in the three-sorted Pas language: it contains ring operations for the first and third sorts, group operation and order relation for the second sort, the valuation map $v$ from the first sort to the second one, and the angular component map $\ac$ from the first sort to the third one. We define the angular component map for $K_A$ and $K_B$ as the map that takes the lowest coefficient in the corresponding power series.
	
	Now we prove an even stronger statement: any first order formula $P(u_1,\ldots,u_n)$ in the Pas language in $n$ variables from the first sort has the same value on $\ovl{\xi}$ and $\ovl{\zeta}$. It follows from~\cite{Pas} that in the theory of henselian valued fields with angular components where the residue field has characteristic zero any formula is equivalent to a boolean combination of formulas of the following types:
	
	$\bullet$  $\ph(\ovl{u})$, where $\ph$ is a quantifier free formula in the language of rings;
	
	$\bullet$ $\psi(v(f_1(\ovl{u})), . . . , v(f_k(\ovl{u})))$, where $\psi$ is a formula in the language of ordered groups and $f_i$ are terms in the ring language.
	
	$\bullet$ $\theta(\ac(g_1(\ovl{u})), . . . , \ac(g_k(\ovl{u})))$, where $\theta$ is a formula in the ring language and $g_i$ are terms in the ring language.

	Moreover, by~\cite[Corollary 3.1.17 ]{MarkerBook} the theory of ordered divisible abelian groups has quantifier elimination; hence in the second item $\psi$ can be taken quantifier free. Now the first two types of formulas have the same value on $\ovl{\xi}$ and $\ovl{\zeta}$ because $\ovl{\xi}$ and $\ovl{\zeta}$ lie in the isomorphic substructures, with isomorphism that takes $\ovl{\xi}$ to $\ovl{\zeta}$ and the formulas are quantifier free. The third type gives the same value because the corresponded angular components are equal.
\end{proof}

\begin{prop} Let $(A,<_A)$ and $(B,<_B)$ be totally ordered sets such that $K_A\simeq K_B$ (as fields). Then $(A,<_A)\simeq (B,<_B)$.\end{prop}
\begin{proof}
	First let us prove that the isomorphism must preserve the valuation ring. Assume it does not. Then the field $K_A$ has two different valuation rings $\O_1$ and $\O_2$ with maximal ideals $\M_1$ and $\M_2$, such that both fraction fields are $\Q$ and $K_A$ is henselian with respect to both valuation. It follows then by \cite[Theorem 4.4.2]{EnglerPrestel} that $\O_1$ and $\O_2$ are comparable, i.e. one ring is contained in the other and the maximal ideals are included in the opposite direction: say, $\O_1\le \O_2$ and $\M_2\le\M_1$. Therefore, the quotient ring $\O_1/\M_2$ has both fraction field and residue field isomorphic to $\Q$, which is only possible if $\O_1/\M_2\simeq \Q$, i.e. $\O_1=\O_2$.
	
	Now for a valued field $K$ with the valuation ring $\O$ one can recover the valuation group as $K^*/\O^*$ and the order can be recovered as $[\xi_1]<[\xi_2]\Leftrightarrow \xi_2/\xi_1\in\O$. Therefore, $\Gamma_A$ and $\Gamma_B$ must be isomorphic as ordered groups.
	
	Given an ordered abelian group $\Gamma$ and elements $g$,$h\in\Gamma$, we write $g\sim h$ if and only if either $|h|\le |g|$ and for some natural number $n$ we have $|g|\le n|h|$; or $|g|\le |h|$ and for some natural number $n$ we have $|h|\le n|g|$, where $|g|=\max\{g,-g\}$. An Archimedean class of $\Gamma$ is an equivalence class of $\sim$. The set of Archimedean classes inherits the order from $\Gamma$.
	
	Note that the sets of Archimedean classes of $\Gamma_A$ and $\Gamma_B$ are precisely $A$ and $B$. Therefore, we have $(A,<_A)\simeq (B,<_B)$.
\end{proof}

\begin{cor}\label{ExmpFields} The pair $K_{\Z}$ and $K_{\Q}$ is an example of countable isotypic but not isomorphic fields.\end{cor}

\begin{cor}\label{ExmpGroups} The pair $\SL(3,K_{\Z})$ and $\SL(3,K_{\Q})$ is an example of countable isotypic but not isomorphic groups. \end{cor}
\begin{proof}
	Since the fields $K_{\Z}$ and $K_{\Q}$ are isotypic, it easily follows that the groups $\SL(3,K_{\Z})$ and $\SL(3,K_{\Q})$ are isotypic. Further since the fields $K_{\Z}$ and $K_{\Q}$ are not isomorphic, it follows from \cite{Bunina_recent} that the groups $\SL(3,K_{\Z})$ and $\SL(3,K_{\Q})$ are not isomorphic.
\end{proof}
\smallskip


\end{document}